\documentclass[a4paper,12pt]{amsart}
\usepackage[public]{optional}
\opt{private}{
\newcommand{\private}[1]{{#1}}  
\usepackage[notref,notcite]{showkeys}

}
\opt{public}{
\newcommand{\private}[1]{}

}
\newcommand{\PL}[1]{\private{\footnote{\textbf{Pascal: }{#1}}}}

\usepackage{amsmath,amssymb,amsfonts,amsthm}
\usepackage{endnotes}
\usepackage{fancyhdr}
\usepackage{lmodern}
\usepackage[utf8]{inputenc}
\usepackage[T1]{fontenc}
\usepackage{enumerate}
\usepackage[all]{xy}
\usepackage{graphicx}
\usepackage[english]{babel}

\newdir{ >}{{}*!/-5pt/\dir{>}}
\newtheorem{theo}{Theorem}[section]    
\newtheorem{de}[theo]{Definition}
\newtheorem{prop}[theo]{Proposition}
\newtheorem{lem}[theo]{Lemma}    
\newtheorem{cor}[theo]{Corollary}

\newtheorem{rem}[theo]{Remark}

\newtheorem{quest}[theo]{Question}

\newtheorem{exemple}[theo]{Exemple}
\newcommand*{\defeq}{\mathrel{\vcenter{\baselineskip0.5ex \lineskiplimit0pt
                     \hbox{\scriptsize.}\hbox{\scriptsize.}}}%
                     =}

\newcommand{\dela}{\partial A}
\newcommand{\delb}{\partial B}
\newcommand{\delt}{\partial T}

\newcommand{\apl}{A_{PL}}

\newcommand{\delw}{\partial W}
\newcommand{\delW}{\partial W}

\newcommand{\bq}{\mathbb{Q}}

\newcommand{\Ho}{\operatorname{H}}
\newcommand{\hoker}{\text{hoker }}

\newcommand{\sn}{s^{-n}\#}

\def\commutatif{\ar@{}[rd]|{\circlearrowleft}}
\def\pullback{\ar@{}[rd]|{pb}}
\def\comhotopie{\ar@{}[rd]|{\sim}}

\newcommand{\calO}{{\mathcal O}}

\newcommand{\qi}{\stackrel{\simeq}\longrightarrow}
\newcommand{\betab}{\xymatrix{\beta \colon B \ar@{->>}[r]& \delb}}

\newcommand{\BQ}{\mathbb{Q}}

\newcommand{\id}{\operatorname{id}}

\title[Pretty rational models]{Pretty rational
  models for Poincar\'e duality pairs}

\author{Hector Cordova Bulens}
\address{H.C.B. and P.L.: IRMP, Universit\'e catholique de Louvain, 2 Chemin du Cyclotron, B-1348 Louvain-la-Neuve, Belgium}
\email{hector.cordova@uclouvain.be}
\author{Pascal Lambrechts}
\email{pascal.lambrechts@uclouvain.be}
\author{Don Stanley}
\address{D.S.: University of Regina, Department of Mathematics\\
College West, Regina, CANADA}%
\email{stanley@math.uregina.ca}%
\thanks{H.C.B. and P.L. are supported by the belgian federal fund PAI Dygest}%
\date{18 May 2015}

\begin{document}
\maketitle
\begin{abstract}
We prove that a large class of Poincar\'e duality pairs admit rational
models (in the sense of Sullivan) of a particularly nice form associated
to some Poincar\'e duality CDGA. These models have applications in
particular to the construction of rational models of configuration
spaces in compact manifolds with boundary. 
\end{abstract}

\PL{All the footnotes tagged \textbf{Pascal} appear in the private version only and are
  comments to the authors. To get a ``public'' version change the
  patameter of the usepackage opt at the beginning of the latex file}
\section{Introduction}
Sullivan theory \cite{Sul:inf} encodes the rational homotopy type of a
simply-connected space of finite type, $X$, into a commutative
differential graded algebra (CDGA), $(A,d_A)$, such that $\Ho(A,d_A)
\cong \Ho^\ast(X;\BQ)$ and which is called a \emph{CDGA model} of $X$
(see Section \ref{s:RHT} fo a quick recapitulation on that theory). In
\cite{LaSt:PD}  we proved that when $X$ is a simply-connected 
Poincar\'e duality space
(the most important example being  a closed manifold),
we can always construct a CDGA model whose underlying algebra
satisfies Poincar\'e duality.  These Poincar\'e duality CDGA models
are often convenient and were used for example in \cite{LaSt:FM2Q} and
in \cite{LaSt:remFMk}
to construct
nice rational models for configuration spaces in closed manifolds,
 or in \cite{FeTh:BVstr}
 to study the Chas-Sullivan product on the free loop
 space
.

The aim of this paper is to exhibit convenient CDGA models for
Poincar\'e duality \emph{pairs} of spaces, like compact manifolds with
boundary.   Such a model should 
 be a CDGA morphism between two CDGAs representing each element of the
 pair.   Our main result is that a very large class of Poincar\'e
 duality pairs admit what we call \emph{pretty models} (Definition
 \ref{D:prettymodel}.)
 More precisely we will show that the following Poincar\'e duality pairs admit such models:
\begin{itemize}
\item even dimensional disk bundles over a simply-connected closed
  manifold relative to their sphere bundles (Theorem \ref{T:diskbdl});
\item   Poincar\'e duality pairs $(W,\delW)$ where $\delW$ is
  2-connected and retracts rationally on its half-skeleton
  (Definition \ref{D:retrhalfsk} and Theorem \ref{T:etruscus});
\item the complement of a subpolyhedron of high codimension in a
  closed manifold (relative  to its natural boundary)
  (end of Section \ref{S:pretty}  and Theorem \ref{T:etruscus} applied to Exemple  \ref{Ex:retract}
  (\ref{Exitem:complpoly}).) 
\end{itemize}

Let us describe roughly the form of these \emph{pretty models} (see
Section \ref{S:pretty} for more details).   
A pretty model for $\delW   \hookrightarrow W$ is a CDGA morphism between mapping cones.
\begin{equation}
\label{E:prettyintr}
\varphi \oplus \id : P \oplus_{\varphi^!} ss^{-n} \#Q \longrightarrow Q\oplus_{\varphi\varphi^!} ss^{-n} \#Q
\end{equation}
where 
\begin{itemize}
\item $P$ is a Poincar\'e duality CDGA in dimension $n=\dim W$ (see
  Definition \ref{D:PDCDGA});
\item $\varphi : P \to Q$ is a CDGA morphism;
\item $ss^{-n} \#Q$ is the $(n-1)-th$ suspension of the linear dual $\#Q=$ \textrm{hom}$_\BQ (Q,\BQ)$;
\item $\varphi^! : s^{-n} \#Q \to P$ is a $P$-dgmodule morphism constructed out of $\varphi$ and the Poincar\'e duality isomorphism on $P$;
\item the CDGA structure on the mapping cones is the semi-trivial one
  described at Section \ref{s:semi_trivial} (which requires that
  $\varphi\varphi^!$ is \emph{balanced} in the sense of Definition \ref{D:balanced}.)
\end{itemize}
In the special case when $\delW = \emptyset$, we have $Q=0$ and we
recover a Poincar\'e duality CDGA model, $P$, for $W$ as in \cite{LaSt:PD}

Note also that the model
\[
Q \oplus_{\varphi\varphi^!} ss^{-n}\#Q
\]
is a Poincar\'e duality CDGA model in dimension $n-1$ for $\delW$.

When $\delW \neq \emptyset$, $W$ is not a Poincar\'e duality space and thus does not admit a Poincar\'e duality CDGA model.    However, often $W$ has a model which is an explicit quotient of a Poincar\'e duality space, as shows the following
\begin{prop}[Corollary \ref{Cor:P/I}]
If $(W,\delW)$ admits a pretty model $(\ref{E:prettyintr})$ and if
$\varphi$ is surjective, then $W$ has a CDGA model
\[
P/I
\]
where $P$ is a Poincar\'e duality CDGA and $I=\varphi^! (s^{-n}\#Q)$ is a diffe\-rential ideal.
\end{prop}

These pretty models should be very convenient in many constructions
in rational homotopy theory on Poincar\'e duality pairs.    In
particular we use
 them in \cite{CLS:compl} and \cite{CLS:FWk}
 to obtain explicit models for the configuration spaces in a manifold with boundary.

Here is the plan of the paper. In Section \ref{s:map_con} we review
quickly basic facts and terminology about rational homotopy theory,
and we define the semi-trivial CDGA structure on some mapping cones.
In Section \ref{S:pretty} we define pretty models for Poincar\'e
duality pairs and we motivate this
definition by the example of the complement of a polyhedron in a
closed manifold. In Section \ref{S:diskbdl} we prove that even
dimensional disk bundles over simply-connected Poincar\'e duality
spaces admit pretty models.
Section \ref{S:PDCDGA} is devoted to the construction of nice
Poincar\'e duality CDGA modelling a given CDGA whose cohomology
satisfies Poincar\'e duality; these results will be needed in the next
section.
We prove in Section
\ref{s:retracthalf} that any Poincar\'e duality pair whose boundary is
$2$-connected and retracts rationally on its half skeletton admits a pretty
model. In the last section we wonder whether every Poincar\'e duality
pair admits a pretty model.

\section{CDGA's,  dgmodules, and semi-trivial CDGA structures on mapping cones}
\label{s:map_con}

\bigskip

\subsection{ Rational homotopy theory} \label{s:RHT}%
In this paper we will use the standard tools and results of rational
homotopy theory,  following the notation and terminology of \cite{FHT:RHT}. Recall that $\apl$ is
the Sullivan-de Rham contravariant functor and that for a 1-connected space of
finite type, $X$,
$\apl(X)$ is a commutative differential graded algebra (CDGA for
short, always non-negatively graded.)
Any CDGA  weakly equivalent to $\apl(X)$ is called a \emph{CDGA model
  of $X$} and it completely encodes the rational homotopy type of $X$.  
Similarly a CDGA model of a map of spaces $X\to Y$ is a CDGA morphism
weakly equivalent to $\apl(f)\colon\apl(Y)\to\apl(X)$.
All our dgmodules and CDGAs are over the field $\bq$.
 A CDGA, $A$, is \emph{connected} if $A^0=\BQ$. A \emph{Poincar\'e
   duality CDGA}
is a connected CDGA whose underlying algebra satisfies Poincar\'e duality (see
Definition \ref{D:PDCDGA} for a precise definition.)

\subsection{Mapping cones and semi-trivial CDGA structures}
\label{s:semi_trivial}
Let $A$ be a CDGA and 
 let $R$ be an $A$-dgmodule. We will denote by $s^k R$ the k-th
 suspension of $R$, i.e. $(s^k R)^p= R^{k+p}$, and for  a map of
 $A$-dgmodules, $f\colon R \to Q$,  we denote by $s^k f$ the  k-th
 suspension of  $f$. 
For example, $s^{-n}\BQ$ is a dgmodule concentrated in degree $n$.
Furthermore, we will use $\#$ to denote the linear dual of a vector space, $\#V=hom(V,\mathbb{Q})$, and $\#f$ to denote the linear dual of a map $f$.  A dgmodule is of \emph{finite type} if it is of finite dimension in every degree.    If $M$ is a dgmodule, we write $M^{> k}=0$ to express that $M^i=0$ for each $i>k$; similarly we will write $M^{\geq k}=0$, $M^{<k}=0$, etc.

If $f\colon Q \to R$ is an $A$-dgmodule morphism, \emph{the mapping cone} of $f$ is the $A$-dgmodule 
\[
C(f)\defeq (R\oplus  sQ, \delta)
\]
defined by $R\oplus sQ$ as an $A$-module and with a differential
$\delta$ such that $\delta (r,sq) = (d_R (r) +f(q), -sd_Q(q))$.   We also write $C(f)=R \oplus_f sQ$.   When $f=0$, we just write $C(0)=R \oplus sQ$.

When
$R=A$, the mapping cone $C(f\colon Q\to A)$  can be equipped with a unique
commutative graded algebra (CGA) structure that extends the algebra
structure on $A$ and  the $A$-dgmodule structure on $sQ$,  and such
that $(s q) \cdot (sq')=0$, for $q,q' \in Q$. We will call this
structure the \emph{semi-trivial structure} on the mapping cone
$A\oplus_fsQ$ (see \cite[Section 4]{LaSt:PE}).      

\begin{de}
\label{D:balanced} Let $A$ be a CDGA. An $A$-dgmodule morphism $f:Q\to A$ is \emph{balanced} if, for each $x,y\in Q$:
\begin{equation}
f(x)\cdot y = x \cdot f(y).\label{Eq:anodyne}
\end{equation} \end{de}

\par\noindent
The importance of this notion comes from the following proposition:
\begin{prop}
\label{p:nagata}
Let $Q$ be an $A$-dgmodule and $f\colon Q \to A$ be an $A$-dgmodule morphism. If $f$ is balanced  
then the mapping cone $C(f)=A\oplus_f sQ$ endowed with the semi-trivial CGA
structure is a CDGA. 
\end{prop}
\begin{proof}
The only non trivially verified condition for $C(f)$ being a CDGA is the Leibniz rule for the differential. 
Let $a,a' \in A$ and $q,q' \in Q$. 
For products of the form $(a,0)\cdot (a',0)$ and of the form $(a,0)\cdot(0,sq)$ the Leibniz rule is verified  because $A$ is a CDGA and $Q$ is an $A$-dgmodule.      For products of the form $(0,sq)\cdot (0,sq')$, by semi-trivality of the CDGA structure of the mapping cone we have to verify that
\[(\delta (0,sq)) \cdot (0,sq') + (-1)^{|q|+1} (0,sq)\cdot(\delta (0,sq'))=0,\]
which is a consequence of the hypothesis that $f$ is balanced.
\end{proof}

\PL{I removed the following corallary and its proof which I guess is useless in the paper:\\
\begin{cor}
Let $f:Q\to A$ be an $A$-dgmodule morphism.  If there exists $r\geq 0$ such that $Q^k=0$ when $k<r$ or $k\geq 2r$, then $C(f)$ equipped with the semi-trivial structure is a CDGA.
\end{cor}
\begin{proof}
The map $f$ is balanced for degree reasons.
\end{proof}
}

\begin{rem}
 In the rest of this paper, when a mapping cone is equipped with a CDGA structure it will be understood that is comes from the semi-trivial structure.
\end{rem}

\section{Pretty models}\label{S:pretty}

In this section we first  describe precisely what we call \emph{pretty
  models}, and next we motivate this definition by showing that these
models arise naturally as models of complements of a subpolyhedron in
a closed manifold.

Suppose given
\begin{itemize}
\item[(i)]  a connected Poincar\'e duality CDGA, $P$, in dimension $n$
  (see Definition \ref{D:PDCDGA} below);
\item[(ii)] a connected CDGA, $Q$;
\item[(iii)] a CDGA morphism, $\varphi : P \to Q$.
\end{itemize}
By definition of a Poincar\'e duality CDGA, there exists an isomorphism of $P$-dgmodules
\begin{equation}
\label{Eq:thetaP}
\theta_P : P \stackrel{\cong}{\longrightarrow}  s^{-n} \# P
\end{equation}
which is unique up to   multiplication by a non zero scalar because $P$ is a free $P$-module generated by 1.   

Consider the composite
\begin{equation}
\label{Eq:phishriek}
\varphi^! : s^{-n} \# Q \stackrel{s^{-n}\# \varphi}{\longrightarrow} s^{-n} \# P \stackrel{\theta^{-1}_P}{\longrightarrow} P,
\end{equation}
which is a morphism of $P$-dgmodule.    It is a \emph{shriek map} or \emph{top degree} map in the sense of \cite{LaSt:PE}, because $\Ho^n (\varphi^!)$ is an isomorphism.
\\
Assume that  
\[
\varphi \varphi^! : s^{-n} \# Q \to Q
\]
is balanced (see Definition \ref{D:balanced}).       Then $\varphi^!$ is also balanced  because the $P$-module structure on $s^{-n} \# Q$ is induced throught $\varphi$.  By Proposition \ref{p:nagata} the mapping cones
\[
{P \oplus_{\varphi^!} ss^{-n} \# Q } \quad \textrm{ and } \quad {Q \oplus_{\varphi\varphi^!}  ss^{-n} \# Q }
\]
are CDGA  and
\begin{equation}\label{Eq:wdf}
\varphi \oplus id : P \oplus_{\varphi^!} ss^{-n} \# Q \longrightarrow Q \oplus_{\varphi\varphi^!} ss^{-n} \# Q
\end{equation}
is a CDGA morphism.

\begin{de}\label{D:prettymodel}
Let $\varphi : P \longrightarrow Q$ be a   CDGA morphism with $P$ a connected Poincar\'e duality CDGA in dimension $n$, consider $\varphi^! : s^{-n} \# Q \longrightarrow P$ defined at (\ref{Eq:phishriek}), and assume that $\varphi\varphi^!$ is balanced.  Then the CDGA morphism
\begin{equation}\label{E:prettydef}
\varphi \oplus id : P \oplus_{\varphi^!} ss^{-n} \# Q \longrightarrow Q \oplus_{\varphi\varphi^!} ss^{-n} \# Q
\end{equation}
is called the \emph{pretty model associated to $\varphi$}.  If moreover $\varphi$ is surjective, we say that it is a \emph{surjective pretty model}.
\end{de}

\begin{prop}
If $(\ref{E:prettydef})$ is a surjective pretty model then the
projection
\[\pi\colon P\oplus_{\varphi^!}ss^{-n}\#Q\stackrel{\simeq}{\longrightarrow}
P/I,\]
where $I:=\varphi^!(s^{-n}\#Q)$, is a quasi-isomorphism of CDGA. 
\end{prop}

\begin{proof}
$I=\varphi^! (s^{-n} \# Q)$   is a differential ideal of $P$ because it is the image of a morphism of $P$-dgmodules.  Since $\varphi$ is surjective, by duality, $\varphi^!$ is injective and we have a short exact sequence
\[
0 \longrightarrow s^{-n} \# Q \stackrel{\varphi^!}{\longrightarrow}  P \stackrel{proj}{\longrightarrow} P/I \longrightarrow 0.
\]
Thus
\[
\pi := (proj, 0) : P \oplus_{\varphi^!} ss^{-n} \# Q \longrightarrow P/I
\]
is a quasi-isomorphism of CDGA.
\end{proof}

\begin{cor}\label{Cor:P/I} 
If a Poincar\'e duality pair $(W,\partial W)$ admits a surjective
pretty model $(\ref{E:prettydef})$, then a CDGA model of $W$ is given
by $P/I$ where $P$ is a Poincar\'e duality CDGA.
\end{cor}
\vspace{5mm}

To motivate the above definition, let us show how pretty models appear naturally as models of the complement of a
subpolyhedron in a closed manifold.   
Let $M$ be a simply-connected  closed triangulated manifold and let $K$ be a
subpolyhedron and assume that $\dim(M)\geq2\dim(K)+3$.   The complement $M\setminus K$ is a open manifold
which is
 the interior of a compact manifold $W$ whose boundary, $\delW$, is
 the boundary of a thickening of $K$ in $M$.  
Let
\[\varphi\colon P\to Q\]
be a CDGA model of the inclusion $K\hookrightarrow M$ where $P$ is a
Poincar\'e duality CDGA model of $M$ and $Q^{\geq n/2-1}=0$ (by Proposition
\ref{DPADGC_alg_surjective},
such a model exists under our high codimension hypothesis.)
Consider the morphism $\varphi^!$ defined as the composite
($\ref{Eq:phishriek}$) which is a shriek map in the sense of
\cite[Definition 5.1] {LaSt:PE}.
By the main result of that paper, \cite[Theorem 1.2] {LaSt:PE}, the
pretty model $(\ref{Eq:wdf})$ is then a CDGA model of the inclusion $\partial
W\hookrightarrow W$ (using the fact that $\varphi\varphi^!=0$, and
hence is balanced, for
degree reasons.) This example show that pretty models appear naturally as
models of complements of a subpolyhedron of high codimension.

Notice also that any compact manifold with boundary, $W$, arises as the
complement of a subpolyhedron in a closed manifold. Indeed we can
consider the double $M:=W\cup_\partial W'$, where $W'$ is a second
copy of $W$. Then $M$ is a closed manifold and $W$ is the complement
of $W'$ in $M$. Of course, $W'$ is not necessarily of high homotopical
codimension, thus the discussion above does not apply stricly to this
example, which explains why this construction does not directly imply that any
compact manifold with boundary admit a pretty model.

Note also that when $W$ is the complement of a high codimension
subpolyhedron in a closed manifold, then the boundary, $\partial W$, 
retracts rationally on its half skeletton (in the sense of Definition
\ref{D:retrhalfsk}, see Example \ref{Ex:retract}), and therefore
Theorem \ref{T:etruscus} gives another proof that this complement
admits a pretty model.

\section{Disk bundles over Poincar\'e duality spaces}
\label{S:diskbdl}
In this section we prove that we can construct explicit pretty models
for the total space of an even-dimensional disk bundle over a closed manifold.
\begin{theo}\label{T:diskbdl}
Let $\xi$ be a vector bundle of even rank over a simply-connected
Poincar\'e duality space.   Then the pair $(D\xi,S\xi)$ of associated (disk, sphere) bundles admits a surjective pretty model.

Moreover this model can be explicitly constructed out of any CDGA model of the base and from the Euler class of the bundle.
\end{theo}

\begin{proof}
Assume that $\xi$ is of rank $2k$ and that the base is a Poincar\'e
duality space in dimension $n-2k$. Recall from \cite{LaSt:PD} or Definition
\ref{D:PDCDGA} the notion of a Poincar\'e duality CDGA.
Let $Q$ be a Poincar\'e duality CDGA model of the base
(which, by \cite{LaSt:PD} or Proposition \ref{P:constrPDCDGA}, exists and can be explictely constructed out of any CDGA model
of the base)
 and let $e\in Q^{2k} \cap \ker d$
 be a representative of the Euler class of the bundle.   Then a CDGA model of the sphere bundle is given by
\begin{equation}
\label{E:QLz}
Q \rightarrowtail (Q \otimes \wedge z, dz = e)
\end{equation}
with $\deg (z) = 2k-1$, and this is also a model of the pair $(D\xi,S\xi)$.

Denote by $\bar z$ a generator of degree $2k$ and define the CDGA
\[
P := \left(\frac {Q\otimes \wedge \bar z} {\left(\bar z^2-e\bar z\right)}, D\bar z=0\right)
\]
which is a Poincar\'e duality CDGA in dimension $n$ (where $\left(\bar z^2-e\bar
  z\right)$ is the ideal in $Q\otimes\wedge \bar z$ generated by this
difference).
As vector spaces we have $P\cong Q\oplus\bar z\cdot Q$.
   Define the CDGA morphism
\[
\varphi : P \longrightarrow Q
\]
by $\varphi (q_1+q_2\bar z)=q_1+e\cdot q_2$, for $q_1$, $q_2 \in Q$.  

We will show that the pretty model associated to $\varphi$ is equivalent to the CDGA morphism (\ref{E:QLz}), which will establish the proposition.   Consider the following diagram
\[
\xymatrix{
s^{-n}\#Q\ar[rr]^{s^{-n}\#\varphi}&&s^{-n}\#P\ar[rr]^{\theta^{-1}_P\,\cong}&&P\\
s^{-2k}Q\ar[u]^{\cong}_{s^{-2k}\theta_Q}\ar[urrrr]_{\Phi^!}
}
\]
where $\theta_Q$ and $\theta_P$ are the Poincar\'e duality
isomorphisms, 
and we set $\varphi^! := \theta^{-1}_P\circ(s^{-n} \# \varphi)$ and $\Phi^! := \varphi^! \circ(s^{-2k} \theta_Q)$.  We prove that we can assume that $\Phi^!$ is given by
\begin{equation}
\label{E:Phi!}
\Phi^! (s^{-2k} q) = q \cdot \bar z.
\end{equation}
Indeed $\Phi^! (s^{-2k}1)=\alpha+\lambda \bar z$ for some $\alpha \in
Q^{2k}$ and $\lambda \in Q^0= \BQ$.   Since $\Phi^!$ is a morphism of
$P$-dgmodules, $\bar z \cdot \Phi^! (s^{-2k}1) = \Phi^! (\bar z \cdot
s^{-2k}1)$ which implies, using that $\bar z \cdot
s^{-2k}1=e\cdot
s^{-2k}1$,
\[
\alpha \bar z + \lambda \bar z^2 = e\alpha + \lambda e \bar z,
\]
and therefore, since $\bar z^2=e\bar z$, $\alpha=0$.  Also $\lambda \neq 0$ because $\Phi^!$ induces an isomorphism in $\Ho^n(-)$, since $s^{-n}\#\varphi$ does.  We can replace the isomorphism $\theta_P$ by $\lambda \cdot \theta_P$ and we get
\[
\Phi^! (s^{-2k}1)=\bar z,
\]
which implies (\ref{E:Phi!}).

A direct computation shows  that $\varphi{\Phi^!}$ is balanced, and hence also $\varphi\varphi^!$.  The pretty model associated to $\varphi$ is isomorphic to
\[
\varphi \oplus \id : P \oplus_{\Phi^!} ss^{-2k} Q \longrightarrow Q \oplus_{\varphi{\Phi^!}} ss^{-2k} Q.
\]
The codomain of $\varphi \oplus \id$ is isomorphic to $(Q\otimes \wedge z, dz=e)$ because $\varphi{\Phi^!} (s^{-2k}1)=e$.
The inclusion of $Q$ in the domain of $\varphi \oplus \id$, 
\[
Q \hookrightarrow P \hookrightarrow P \oplus_{\Phi^!} ss^{-2k} Q,
\]
is clearly a quasi-isomorphism.  Thus $\varphi \oplus \id$ is
weakly equivalent to (\ref{E:QLz}), which is a model of $(D\xi,S\xi)$, and the proposition is proved.
\end{proof}

\section{Poincar\'e duality CDGA's}
\label{S:PDCDGA}
In this section we recall that any $1$-connected  closed manifold
admits a Poincar\'e duality CDGA-model, and we prove some relative
version of that result (Proposition \ref{DPADGC_alg_surjective}.)
This will be used in the next section to build more pretty models.

Any $1$-connected closed $n$-dimensional manifold, $M$, satisfies Poincar\'e 
duality in degree $n$, which means that there is an isomorphism
\[\Ho^*(M;\BQ)\cong\#\Ho^{n-*}(M;\BQ)\] of $\Ho^*(M;\BQ)$-modules.
In \cite{LaSt:PD} we proved that Poincar\'e duality holds not only in
cohomology but also  on some
CDGA model of $M$. To make this precise, we review the following 
\begin{de}
\label{D:PDCDGA}
An  \emph{oriented Poincar\'e duality  CDGA in dimension $n$}, or PDCDGA, is a connected CDGA of finite type, $P$,  equipped with an isomorphism of $P$-dgmodules
\begin{equation}
\label{E:thetaP}
\theta_P\colon P\stackrel{\cong}{\to}\sn P.
\end{equation}
\end{de}
\begin{rem}
Since $P$ is a free $P$-dgmodule generated by a single element, the
isomorphism  $\theta_P$ of (\ref{E:thetaP}) is unique up to a
multiplication by a non-zero scalar.
When this isomorphism is not specified, we talk of a \emph{Poincar\'e
  duality CDGA} (dropping the adjective \emph{oriented}.) 
\end{rem}
\begin{rem}
It is easy to check that Definition \ref{D:PDCDGA} is equivalent to
{\cite[Definition 2.2]{LaSt:PD}}. Indeed the orientation
$\epsilon\colon P\to s^{-n}\BQ$
of {\cite[Definition 2.2-2.3]{LaSt:PD}} is obtained by $\epsilon:=\theta_P(1)$, where
$1\in\BQ$. Conversely, the isomorphism $\theta_P$  is obtained from
the orientation $\epsilon$ 
by  $(\theta_P(y))(x):=\epsilon(x.y)$.\\
\end{rem}

The main result of \cite{LaSt:PD} is that any CDGA whose cohomology
is $1$-connected and satisfies Poincar\'e duality is weakly equivalent 
to some Poincar\'e duality  CDGA.
The aim of this section is  to prove the
follo\-wing   relative version of that result:
\begin{prop}
\label{DPADGC_alg_surjective}
Let $\psi \colon A\to B$ be a morphism of CDGA such that
$\Ho(A)$ \ is  \ a  \ $1$-connected  \ Poincar\'e  \  duality  \ algebra  \ in  \ dimension  \ $n$,
$\Ho^{\geq \frac n2-1}(B)=0$, $\Ho(B)$ is $1$-connected and of finite type,
and $\Ho^2(\psi)$ is surjective.\\
Then $\psi$ is equivalent to some surjective CDGA morphism
\[\varphi\colon P \twoheadrightarrow Q\]
such that $P$ is a $1$-connected Poincar\'e duality CDGA in
dimension $n$, $Q$ is 1-connected, and $Q^{\geq \frac n2-1}=0$. Moreover the morphism
$\varphi$ can be constructed explicitely out of the morphism $\psi$.   Also, if $B$ is 1-connected and $B^{\geq \frac n2 -1}=0$ then we can take $Q=B$.
\end{prop}
A key ingredient to prove this proposition is the following:
\begin{prop}
\label{P:constrPDCDGA}
Let $A$ be a CDGA such that 
$\Ho(A)$ is a  Poincar\'e duality algebra in dimension $n$.
Assume moreover that $n\geq7$, $A$ is $1$-connected of finite type, 
and $A^2\subset\ker d$.\\
Then there exists a CDGA quasi-isomorphism
\[\lambda\colon A\stackrel {\simeq}{\longrightarrow } P\]
such that $P$ is a   Poincar\'e duality CDGA in
dimension $n$ and $\lambda$ is an isomorphism in degrees $<\frac n2-1$.
\end{prop}
This   proposition is an improvement of the main result of
\cite{LaSt:PD} in the sense that the quasi-isomorphism $\rho$ to the
Poincar\'e duality CDGA is an isomorphism below about half the
dimension.

If we take for granted Proposition \ref{P:constrPDCDGA}  then we can prove
Proposition \ref{DPADGC_alg_surjective} as follows.
\begin{proof}[Proof of Proposition \ref{DPADGC_alg_surjective}]
If $n\leq6$ then, since $\Ho^{\geq \frac n2-1}(B)=0$ and $\Ho(B)$ is
$1$-connected, we have $\Ho(B)=\BQ$ and the proposition is a consequence
of the main result of \cite{LaSt:PD} by taking $Q=\BQ$.

Assume now that $n\geq7$.
By passing to Sullivan models it is easy to see that $\psi$ is
equivalent to a surjective morphism between $1$-connected finite type
CDGA's. Thus without loss of generality we assume that $\psi\colon
A\to B$ is
already like that. Moreover, since $\Ho^{\geq \frac n2-1}(B)=0$, by moding
out $B$ by a suitable  acyclic ideal we get a surjective
quasi-isomorphism 
\[\pi\colon B\stackrel {\simeq}{\twoheadrightarrow } Q\]
where  $Q^{\geq \frac n2-1}=0$.

By Proposition \ref{P:constrPDCDGA}, there is a quasi-isomorphism 
$\lambda\colon A\stackrel {\simeq}{\twoheadrightarrow} P$ which is an
isomorphism in degrees $< n/2-1$ and such that $P$ is a Poincar\'e
duality CDGA. Since $(\ker\lambda)^{<\frac n2-1}=0$ and $Q^{\geq \frac n2-1}=0$, the
morphism $\pi\psi$ extends along $\lambda$ into the desired morphism
$\varphi\colon P \twoheadrightarrow Q$.
\end{proof}

The rest of this section is devoted to the proof of Proposition
\ref{P:constrPDCDGA}. Since the techniques used here will not appear
in the rest of this paper, the reader can safely jump to the next
section if he wishes.     The proof  is based on techniques of \cite{LaSt:PD} and we
assume that the reader is familiar with the notation and proofs of
that paper. In particular, recall that an \emph{orientation} (in
degree $n$) of a CDGA
$(A,d_A)$ is a chain map
\[\epsilon\colon A\to s^{-n}\BQ\]
that is surjective in cohomology \cite[Definition 2.3]{LaSt:PD}.  We then say that $(A,d,\epsilon)$ is an \emph{oriented CDGA}.   Note that $\epsilon$ is a chain map if and only if $\epsilon(d(A^{n-1}))=0$.   
The differential ideal of \emph{orphans} of this oriented CDGA is (\cite[Definition 3.1 and  Proposition 3.2]{LaSt:PD})
\[\calO=\calO(A,\epsilon):=\{a \in A\textrm{ such that }\forall b\in
A:\epsilon(a\cdot b)=0\}.\]
The main interest of the notion of orphans is that when $A$ is 1-connected of finite type and $\Ho(A)$ is a Poincar\'e duality algebra in degree $n$ then $P=A/\calO$ is  
a Poincar\'e duality CDGA in degree $n$ \cite[Proposition
3.3]{LaSt:PD}. Moreover, P is quasi-isomorphic to $A$ when $\calO$ is
acyclic.     The strategy of the proof of Proposition \ref{P:constrPDCDGA} is to build a Sullivan extension $(A,d) \rightarrowtail (\hat A := A \otimes \wedge V,\hat d)$ such that $V^{< \frac n2 -1}=0$ and its ideal of orphans $\hat \calO$ is acyclic and without elements of degree $<\frac n2 -1$.  Then $\lambda : A \longrightarrow  {\hat A}/{\hat\calO}$ will be the desired quasi-isomorphism to a PDCDGA.
For the sake of the proof we need the following definition:
\begin{de}
Let $(A,d,\epsilon)$ be an oriented CDGA.
\begin{itemize}
\item[(i)] 
The oriented CDGA $(A,\epsilon)$ has \emph{no
  orphans in degrees $\leq p$} if $\left(\calO(A,\epsilon)\right)^i=0$ for
$i\leq p$.
\item[(ii)]
 An \emph{acyclic oriented
  Sullivan extension} is a Sullivan extension
\[\xymatrix{(A,d)\ar@{ >->}[r]^-{\simeq}&(\hat A:=A\otimes\wedge V,\hat d)}\]
that is a quasi-isomorphism and is equipped with an orientation
$\hat\epsilon\colon\hat A\to s^{-n}\BQ$ that extends $\epsilon \colon A \to s^{-n} \BQ$.\\
\item[(iii)]
The acyclic oriented Sullivan extension (ii) \emph{adds no
  orphans in degree $\leq q$} if
\[\left(\calO(\hat A,\hat \epsilon)\right)^i\subset \left(\calO( A,
  \epsilon)\right)^i\quad\textrm{ for }i\leq q.\]
\item[(iv)] The acyclic oriented Sullivan extension $(ii)$ \emph{adds no
  generators in degree $\leq m$} if
\[V^i=0\textrm{ for }i\leq m.\]
\end{itemize}
\end{de}

Proposition \ref{P:constrPDCDGA}  will be a consequence of an
inductive application of the following two lemmas.
\begin{lem}\label{L:acyclicorphans}
Let $(A,d,\epsilon)$ be an oriented CDGA of finite type that is
$1$-connected, such that $A^2\subset\ker(d)$ and $\Ho(A,d)$ is a
Poincar\'e duality algebra in degree $n\geq7$.\\
Then $(A,d,\epsilon)$ admits an acyclic oriented Sullivan extension
that adds no orphans in degree $\leq\frac n2-1$, adds no generators in
degrees $<\frac n2-1$, and whose set of orphans is acyclic.
\end{lem}
\begin{proof}
The set of
orphans of  $(A,d,\epsilon)$ is $\frac n2$-half-acyclic (see
\cite[Definition 3.5]{LaSt:PD} and the remark after).
Since, by hypothesis, $(A,d,\epsilon)$ satisfies  \cite[(4.1)]{LaSt:PD} we can apply \cite[Proposition 5.1]{LaSt:PD} iteratively for all integers $k$ ranging
from $\lceil{\frac n2+1}\rceil$ up to $n+1$. More precisely, at each step we construct the extension
described at \cite[Section 4]{LaSt:PD} for the integer $k\geq \frac n2+1$.
This is an acyclic Sullivan extension defined at the equation
\cite[(4.4)]{LaSt:PD} which is oriented by \cite[Lemma
4.5]{LaSt:PD}. The set of orphans in this extension is $k$-half-acyclic by
\cite[Proposition 5.1]{LaSt:PD}.

The new generators of lowest degrees in the extension \cite[(4.4)]{LaSt:PD} are the $w_i$'s of
degree $k-2\geq \frac n2-1$. Thus the extension adds no generator of
degrees $<\frac n2-1$. 

Since by \cite[(4.1)]{LaSt:PD}, $d\gamma_i=\alpha_i$ and since, by \cite[(4.2)]{LaSt:PD}, $\alpha_i$ is not the boundary of an orphan in $A$, there exists $\xi_i \in A$ such that $\epsilon (\gamma_i \xi_i)\neq 0$.   By \cite[(4.5)(ii)]{LaSt:PD}, $\hat\epsilon (w_i d(\xi_i)) = \pm \epsilon(\gamma_i\xi_i)\neq 0$ and therefore $w_i$ is not an orphan in $\hat A$.  
Thus the extension adds no orphans in degrees $\leq k-2$, hence in 
degrees $\leq \frac n2-1$.

When we reach $k=n+1$, the set of orphans
is $(n+1)$-half acyclic, and therefore is acyclic by \cite[Proposition
3.6]{LaSt:PD}.
\end{proof}

\begin{lem}
\label{L:noorphans<=p}
Let $(A,d,\epsilon)$ be as in Lemma \ref{L:acyclicorphans}. Assume that
its set of orphans is acyclic and that there are no orphans in degrees
$<p$ for some integer $1 \leq p<\frac n2-1$. 
Then $(A,d,\epsilon)$ admits an acyclic oriented Sullivan extension 
with no orphans in degrees $\leq p$
and which adds  no generators in
degrees $\leq\frac n2$.
\end{lem}
\begin{proof}
Let $\calO$ be the ideal of orphans in $(A,d,\epsilon)$.
Since $\calO$ is acyclic and $\calO^{<p}=0$, we have
$\calO^p\cap\ker(d)=0$.
Let $\{x_1,\dots,x_r\}$ be a basis of $\calO^p$.
Consider the acyclic Sullivan extension
\[\hat A:=(A\otimes \wedge(u_1,\dots,u_r,\bar u_1,\dots,\bar u_r),\hat
d)\]
with $\deg(u_i)=n-p-1$, $\deg(\bar u_i)=n-p$, $\hat d(u_i)=\bar u_i$
and $\hat d(\bar u_i)=0$.
We extend the orientation $\epsilon$ into an orientation
$\hat\epsilon$ of $\hat A$ as follows.
 Let $S$ be a supplement space  of $\calO^p\oplus(A^p\cap\ker(d))$ in $A^p$.
Let $T$ be a supplement space of $d(\calO^p)\oplus d(S)$ in $A^{p+1}$.
Since $n-p-1>\frac n2$ there is a unique degree $0$ map 
\[\hat\epsilon\colon\hat A\to s^{-n}\BQ\]
extending $\epsilon$ and such that
\[
\begin{cases}
\hat\epsilon(\bar u_i\cdot x_j)&=\delta_{ij}\quad\textrm{where
  $\delta_{ij}$ is the Kronecker symbol}\\
 \hat\epsilon(\bar u_i\cdot \ker(d))&=0\\
 \hat\epsilon(\bar u_i\cdot S)&=0\\
 \hat\epsilon( u_i\cdot d(x_j))&=(-1)^{n-p}\delta_{ij}\\
 \hat\epsilon( u_i\cdot d(S))&=0\\
 \hat\epsilon( u_i\cdot T)&=0.
\end{cases}
\]
Then one computes that $\hat\epsilon(d(\hat{A}^{n-1}))=0$, and hence $\hat\epsilon$ is an
orientation.

This extension adds no generators in degrees $<n-p-1$, and hence no
generators in degrees $\leq \frac n2$ (because $p<\frac n2-1$.) For the same reasons it adds no orphans in degrees $\leq p$. Moreover all the degree $p$ orphans of
$A$, which are linear combinations of 
$x_1,\dots,x_r$, are not orphans anymore in $\hat A$ since
$\hat\epsilon(\bar u_i\cdot x_j)=\delta_{ij}$. Thus $\hat A$ has no orphans in
degree $\leq p$.
\end{proof}

\begin{proof}[Proof of Proposition \ref{P:constrPDCDGA}]
Let $p_{\max}$ be the largest integer $<\frac n2-1$.
Apply   Lemmas \ref{L:acyclicorphans}, \ref{L:noorphans<=p}, and \ref{L:acyclicorphans} again, successively for $p=1,2,\dots,p_{\max}$.
This gives by composition an acyclic oriented Sullivan extension $\hat
A$ with an acyclic ideal of orphans, with no orphans in
degrees
$<\frac n2-1$, and with no generators added in degree $<\frac n2-1$.
Therefore the composite
\[
\xymatrix{\lambda\colon A\,\ar@{>->}[r]^-{\simeq}&\hat A \ar@{->>}[r]^-{\simeq}&
\frac{\hat A}{\calO(\hat A,\hat\epsilon)}}
\]
is a quasi-isomorphism, and an isomorphism in degrees $<\frac n2-1$, to
a Poincar\'e duality CDGA.
\end{proof}

\section{Poincar\'e duality spaces that retract  rationally on their half-skeleton}
\label{s:retracthalf}

\bigskip

In this section we consider a quite large class of Poincar\'e duality
pairs $(W,\partial W)$ that admit a pretty model. Indeed the only
restriction is on the boundary $\partial W$ which should retracts
rationally on its half-skeleton as explained in the following definition.
\begin{de}\label{D:retrhalfsk}
Let $M$ be a simply-connected Poincar\'e duality space in dimension $n-1$.    We say that it \emph{retracts rationally on its half-skeleton} if there exists a morphism of connected CDGAs
\[
\rho : Q \longrightarrow A 
\]
such that
\begin{itemize}
\item[(i)] $A$ is a CDGA model of $M$,
\item[(ii)]   $\Ho^{\geq n/2-1}(Q)=0$, and
\item[(iii)]    $\Ho^k (\rho)$ is an isomorphism for $k \leq n/2$.
\end{itemize}
\end{de}

\begin{rem}
The terminology comes from the fact that the conditions of the definition imply that the realization of $\rho$ can be thought of as  a retraction of $M$ on a skeleton of half the dimension, as it is clear from Diagram ($\ref{Diag:pi}$) in the proof of Proposition \ref{p:wdfhalfsk}.   
\end{rem}

\begin{rem}
Poincar\'e duality and (ii)-(iii) in the previous definition imply that $M$ has no cohomology about the middle dimension.   More precisely, if $n$ is even then $\Ho^{n/2-1}(M)=\Ho^{n/2}(M)=0$, and if $n$ is odd then $\Ho^{(n-1)/2}(M)=0$.
\end{rem}

\begin{exemple}\label{Ex:retract}
\begin{enumerate}
\item Consider the total space $W$ of a $d$-dimensional disk bundle
  over a closed manifold of dimension $<d-1$. Then $\partial W$
  retracts rationally on its half-skeleton as one checks by building a
  model of the sphere bundle.
\item \label{Exitem:complpoly} Let $K$ be a compact polyhedra embedded
  in a closed manifold $M$ of dimension $n$. Assume that $K$ and $M$
  are $1$-connected and $M\geq 2\dim K + 3$. Let $T$ be a regular
  neighborhood of $K$ in $M$. Then $N \defeq \delt$ retracts rationally
  on its half-skeleton. Indeed, \cite[Theorem 1.2]{LaSt:PE} gives a
  model of $\delt$ of the form $Q\oplus sD$ where $Q$ is a model of
  $K$ and $D\simeq \sn Q$, and the conclusion follows.
\item As a special case of the previous example consider a $1$-connected polyhedron $K$ embedded in $S^n=\mathbb{R}^n \cup \{\infty\}$ with $n\geq 2 \dim K +3$. Then the boundary of a thickening of $K$ in $S^n$ retracts rationally on its half-skeleton.
\end{enumerate}
\end{exemple}

\begin{prop}\label{p:wdfhalfsk}
Let $M$ be a simply-connected Poincar\'e duality space in dimension $n-1$.    Then it retracts rationally on its half-skeleton if and only if there exists a connected CDGA, $Q$, such that
\begin{itemize}
\item[(i)]   $Q^{\geq n/2-1} = 0$, and
\item[(ii)]    $Q\oplus ss^{-n}\#Q$ is a CDGA model of $M$.
\end{itemize}
\end{prop}

\begin{proof}
It is clear that (i) and (ii) imply that $M$ retract rationally on its half-skeleton.
\\
Let us prove the converse. Let
\[
\rho' : Q' \longrightarrow A'
\]
be a morphism between connected CDGAs that satisfies (i)-(iii) of Definition \ref{D:retrhalfsk} (with the added decoration ``prime'').   Consider a minimal Sullivan extension $\hat \rho$:
\[
 \xymatrix{Q'\ar@{ >->}[r]^-{\hat \rho}&(Q'\otimes\wedge V,D')} \stackrel{\simeq}{\longrightarrow} A'
 \]
that factors $\rho'$.    Let $h$ be the integer such that $n=2h$ or $n = 2h+1$.   Since $\Ho^{\leq h}(\rho)$ is an isomorphism and $\Ho^{>h}(Q')=0$, by minimality we have   $V^{\leq h}=0$.     Since $Q'$ is connected and $\Ho^{\geq n/2-1}(Q')=0$, there exists an acyclic ideal $J\subset Q'$ such that $Q^{'\geq n/2-1}\subset J$.     Set $Q := Q'/J$ and consider the push-out of CDGAs
\[\xymatrix{Q' \ar@{ >->}[r]^-{\rho'} \ar[d]_{\simeq} & (Q' \otimes \Lambda V, D')\ar[d]^{\simeq} 
\\
Q \ar@{ >->}_-{\rho}[r] 
\ar@{}[ur]|{pushout}
& (Q \otimes \Lambda V, D)}
\]
and set $A := (Q \otimes \wedge V, D)$.   Note that $\rho$ is an isomorphism in degrees $\leq h$.  Also it endows $A$ with the structure of a $Q$-dgmodule.
\\
Let $S$ be a supplement of $A^h \cap \ker d$ in $A^h$ and set
\[
I := S \oplus A^{>h},
\]
which is an ideal since $A$ is connected.  For degree reasons and since $\Ho^{\leq h}(\rho)$ is an isomorphism, the composite 
\[
Q \stackrel{\rho}{\longrightarrow} A \stackrel{proj}{\longrightarrow} A/I
\]
is a quasi-isomorphism.
\\
By the lifting lemma \cite[Proposition 14.6]{FHT:RHT}, in the following diagram 
\begin{equation} 
\label{Diag:pi}
\xymatrix{Q \ar@{=}[r]\ar@{ >->}_{\rho}[d] & Q\ar[d]_{\simeq} \\
A \ar@{-->}[ur]^{\pi} \ar[r]_{proj} & A/I}
\end{equation}
we get a CDGA morphism $\pi$ that makes the upper left triangle commute and the lower right triangle commute up to homotopy,   in other words, $Q$ is a retract of $A$.

Since $\Ho(A)$ satisfies Poincar\'e duality in dimension $n-1$, there exists a quasi-isomorphism of $A$-dgmodules, hence of $Q$-dgmodules, 
\[
\theta : A \stackrel{\simeq}{\longrightarrow}  ss^{-n} \# A,
\]
and we have the diagram
\[\xymatrix{A \ar[r]^-{\simeq}_-{\theta}&s\sn A\ar[d]^{s\sn \rho} \\
Q\ar[u]^{\rho}& s\sn Q}
\]Set $\lambda = (ss^{-n} \# \rho) \circ \theta$ which is a morphism of $Q$-dgmodules.   Since $\rho$ induces an isomorphism in homology in degrees $\leq n/2$, we get that $ss^{-n} \# \rho$, and hence $\lambda$, induces an isomorphism in homology in degrees $\geq n/2-1$.

Consider the $Q$-dgmodule morphism
\[
\epsilon = (\pi,\lambda) : A \longrightarrow Q \oplus ss^{-n} \# Q.
\]
For degree reasons and since $\pi$ (respectively, $\lambda$) induces isomorphism in homology below (respectively, above) degree $n/2$, we get that $\epsilon$ is a quasi-isomorphism.   

We prove that $\epsilon$ is a morphism of algebras.    Let $a$, $a' \in A$.   If $\deg(a) \leq h$, then, since $\rho$ is an isomorphism in that degree, the multiplication by $a$ is determined by the $Q$-module structure, and since $\epsilon$ is of $Q$-module we get that $\epsilon(a\cdot a')=\epsilon(a)\cdot\epsilon(a')$.   The same arguments work if $\deg(a') \leq h$.  If both $a$ and $a'$ are of degrees $\geq h+1$, then $\deg(a\cdot a')>n$ and then $\epsilon(a\cdot a')=0=\epsilon(a)\cdot \epsilon(a')$, for degree reasons.
\\
Thus $\epsilon$ is a CDGA quasi-isomorphism and the proposition is proved.
\end{proof}

\begin{theo}\label{T:etruscus}
Let $W$ be a simply-connected compact manifold with boundary $\delw$.
Assume that $\partial W$ is 2-connected and retracts rationally on its
half-skeleton. Then $(W,\partial W)$ admits a surjective pretty model.
\end{theo}

\begin{proof}
Since $\delw$ retracts rationally on its half-skeleton, by Proposition
\ref{p:wdfhalfsk} there exists a connected CDGA, $Q$, of finite type such that $Q^{\geq n/2-1}=0$ and   $\apl(\delw) \simeq Q\oplus s\sn Q$.  Since $\delw$ is 2-connected, we can assume that $Q$ is 2-connected.  Then there exists a 1-connected CDGA model  $R$ of $\apl(W)$ and a surjective morphism 
\[\psi \colon R\twoheadrightarrow Q\oplus s\sn Q\]
that is a model of $\apl(W) \to \apl(\delw)$.

Consider the following pullback diagram in CDGA 
\begin{equation}
\label{etruscus:pb}
\xymatrix{ P' \ar[r]^-{\bar\psi} \ar[d]^{i}& Q\ar@{^{(}->}[d]^{\iota} \\
R \ar@{->>}[r]^-{\psi} 
\ar@{}[ur]|{pullback}& Q\oplus s\sn Q, }
\end{equation}
where $\iota$ is the obvious inclusion.  This pullback is a homotopy pullback because $\psi$ is surjective.

We first prove that $P'$ satisfies Poincar\'e duality in cohomology.    We say that a CDGA morphism $\alpha \colon A \to \dela$ satisfies \emph{relative Poincar\'e duality in cohomology} in dimension $n$, if there exists an isomorphism of $\Ho(A)$-modules 
\[\Ho(\hoker \alpha) \cong \sn \Ho(A).
\]
The morphism $\psi$ satisfies relative Poincar\'e duality because it
is a model of the inclusion $\delw \hookrightarrow W$. Moreover
$\hoker (\iota)$ is weakly equivalent to $s^{-n} \#Q$ as a
$Q$-dgmodule, hence    $\iota$ also satisfies relative Poincar\'e
duality in dimension $n$.    Since $\psi$ and $\iota$ satisfy relative
Poincar\'e duality in degree $n$, a  similar argument than the one
used in \cite[Theorem 2.1]{Wal:PC1} implies that $\Ho(P')$ is a Poincar\'e duality algebra in degree $n$. 

By Proposition \ref{DPADGC_alg_surjective} we can factorize $\bar \psi$ in CDGA as follows
\[\xymatrix{P'\ar[rr]^{\bar \psi} \ar[rd]_{\lambda}^{\simeq}&&Q \\
&P.\ar@{->>}[ru]_\varphi& }
 \]
where $\varphi \colon P \twoheadrightarrow Q$ satisfies 
\begin{itemize}
\item[(i)] $P$ is a connected  Poincar\'e duality CDGA in dimension $n$,
\item[(ii)] $\varphi$ is surjective, and
\item[(iii)] $Q^{\geq n/2-1}=0$.
\end{itemize}
 Since \eqref{etruscus:pb} is a homotopy pullback diagram,  $\hoker (i)$ is weakly equivalent as a $P'$-dgmodule to $\hoker (\iota) \simeq \sn Q$.  Therefore, there exists  a cofibrant $P'$-dgmodule $D$ with weak equivalences 
\[\xymatrix{\hoker (i) &\ar[l]_-{\simeq}^-{\gamma'} D \ar[r]^-{\simeq}_-{\gamma} & \sn Q.}\]
 
Set
\[\varphi^! \defeq \theta^{-1}_P \circ \sn \varphi \colon \sn Q \to \sn P\cong P,\]
consider the natural map $l \colon \hoker (i) \to P'$, and recall that $\lambda \colon P' \to P$ is the quasi-isomorphism used to factorize the morphism $\psi$ through $P$. We can verify that $\Ho^n (\varphi^!)$ is an isomorphism, hence $\varphi^!$ is a shriek map or top degree map in the sense of \cite[Section 5]{LaSt:PE}. Also, the fact that $\Ho^n(R)=\Ho^n(W)=0$ implies that $\Ho^n(l)$ is also an isomorphism, and hence $\lambda\circ l$ is also a shriek map. Therefore,   $\varphi^!\circ \gamma$ and   $\lambda\circ l\circ \gamma'$ are both shriek maps.   The unicity of shriek maps up to multiplication by a scalar (\cite[Proposition 5.6]{LaSt:PE}) implies (multiplying $\gamma$ by a non-zero scalar if necessary), that the following diagram of $P'$-dgmodules commutes up to homotopy
\begin{equation}
\label{d:mapco}
\xymatrix{ \hoker i \ar[r]^l & P' \ar[r]^{\bar \psi}\ar[dd]_{\lambda} & Q   \ar@{=}[dd]\\
D \ar[u]_{\gamma'}^{\simeq} \ar[d]^{\gamma}_{\simeq} && \\
 \sn Q \ar[r]^-{\varphi^!}\ar@{}[ruu]|{\sim} & P\ar[r]^{\varphi} & Q.}
\end{equation}

Composition induces a $P'$-dgmodules morphism between the mapping cone of the morphism $l$ and the mapping cone of the morphism $\bar \psi \circ l$:
\[P'\oplus_{l} s \  \hoker (i) \to Q \oplus_{\bar\psi l} \ s \  \hoker (i).\]
The homotopy commutative  Diagram \eqref{d:mapco} implies that this last morphism is weakly equivalent as a $P'$-dgmodules to the morphism 
\[ P\oplus_{\varphi^!} s\sn Q \to Q\oplus_{\varphi\varphi^!} s\sn Q .
\]

From Diagram \eqref{etruscus:pb} we get a quasi isomorphism $P'\oplus_l \ s \ \hoker (i) \qi R$. Therefore,  the morphism \[P'\oplus_{l} \  s \ \hoker (i) \to Q \oplus_{\bar\psi l} \ s \ \hoker (i)
\] 
is weakly equivalent  to the morphism $\psi : R\to Q\oplus s\sn Q$ which is a model of $\apl (W) \to \apl(\delw)$.  Standard techniques  show that these morphisms are weakly equivalent as CDGA morphism. Since $\varphi\varphi^! = 0$ for degree reasons, we deduce that the morphism 
\[
\varphi \oplus id : P\oplus_{\varphi^!} s\sn Q \to Q\oplus s\sn Q\]
is   a CDGA model of  $\delw \hookrightarrow W$. 
\end{proof}

Interestingly enough the hypothesis in Theorem \ref{T:etruscus} is
only on the boundary $\partial W$. This implies that is $\partial W$
satisfies the hypothesis, then any other pair $(W',\partial W')$ with
$\partial W'=\partial W$ will also admits a pretty model. One can get
many such manifolds $W'$ by performing a connected sum in the
\emph{interior} of $W$ with a simply-connected closed $n$-manifold $N$. Or more
generally  we can modify $W$ by surgeries on its interior, which does
not change its boundary. We can thus perform such surgeries on the manifolds from
Example \ref{Ex:retract} to get many other examples.
 
\section{An open question}

We finish this article by asking whether every Poincar\'e duality pair admits a
pretty model. 
  This would implies that every boundary manifold $\delW$ admits a
  model of the form $Q\oplus_{\varphi\varphi^!} ss^{-n} \#Q$, which is
  a very special form of a Poincar\'e duality CDGA.   Thus a
  preliminary algebraic question might be the following 
\begin{quest}\label{Q:end}
Let $(A,d)$ be a CDGA whose homology satisfies Poincar\'e duality in dimension $n-1$ and whose signature is 0.   Does it always exist a CDGA, $Q$, and a balanced $Q$-dgmodule morphism
\[
\Psi : s^{-n} \#Q \longrightarrow Q
\]
such that $(A,d)$ is quasi-isomorphic to
\[
Q\oplus_\Psi ss^{-n} \#Q\quad ?
\]
\end{quest}
A positive answer to this question would be an interesting
reinforcement of the main result of \cite{LaSt:PD}.   Note that we
cannot drop the hypothesis of having zero signature since this is
clearly the case for the CDGA  $Q\oplus_\Psi ss^{-n} \#Q$. Of course a
boundary manifold as $\partial W$ is always of zero
signature. Conversely any rational Poincar\'e duality space with signature $0$
appears as the boundary in some Poincar\'e duality pair, and hence if
every Poincar\'e duality pair admits a pretty model then the answer to
Question \ref{Q:end} will be affirmative.

\bibliographystyle{amsplain}

\begin{thebibliography}{10}

\bibitem{CLS:FWk}
Hector Cordova~Bulens, Pascal Lambrechts, and Don Stanley, \emph{Dgmodule
  models of configurations spaces in a manifold with boundary}, in preparation.

\bibitem{CLS:compl}
\bysame, \emph{Rational models of the complement of a subpolyhedron in a
  manifold with boundary}, preprint.

\bibitem{FHT:RHT}
Yves F{\'e}lix, Stephen Halperin, and Jean-Claude Thomas, \emph{Rational
  homotopy theory}, Graduate Texts in Mathematics, vol. 205, Springer-Verlag,
  New York, 2001. \MR{MR1802847 (2002d:55014)}

\bibitem{FeTh:BVstr}
Yves F{\'e}lix and Jean-Claude Thomas, \emph{Rational {BV}-algebra in string
  topology}, Bull. Soc. Math. France \textbf{136} (2008), no.~2, 311--327.
  \MR{2415345 (2009c:55015)}

\bibitem{LaSt:FM2Q}
Pascal Lambrechts and Don Stanley, \emph{The rational homotopy type of
  configuration spaces of two points}, Ann. Inst. Fourier (Grenoble)
  \textbf{54} (2004), no.~4, 1029--1052. \MR{MR2111020 (2005i:55016)}

\bibitem{LaSt:PE}
\bysame, \emph{Algebraic models of {P}oincar\'e embeddings}, Algebr. Geom.
  Topol. \textbf{5} (2005), 135--182 (electronic). \MR{MR2135550 (2006g:55012)}

\bibitem{LaSt:PD}
\bysame, \emph{Poincar\'e duality and commutative differential graded
  algebras}, Ann. Sci. \'Ec. Norm. Sup\'er. (4) \textbf{41} (2008), no.~4,
  495--509. \MR{2489632 (2009k:55022)}

\bibitem{LaSt:remFMk}
\bysame, \emph{A remarkable {DG}module model for configuration spaces}, Algebr.
  Geom. Topol. \textbf{8} (2008), no.~2, 1191--1222. \MR{2443112 (2009g:55011)}

\bibitem{Sul:inf}
Dennis Sullivan, \emph{Infinitesimal computations in topology}, Inst. Hautes
  \'Etudes Sci. Publ. Math. \textbf{47} (1977), 269--331 (1978). \MR{MR0646078
  (58 \#31119)}

\bibitem{Wal:PC1}
C.~T.~C. Wall, \emph{Poincar\'e complexes. {I}}, Ann. of Math. (2) \textbf{86}
  (1967), 213--245. \MR{0217791 (36 \#880)}

\end{thebibliography}
\def\cprime{$'$}
\providecommand{\bysame}{\leavevmode\hbox to3em{\hrulefill}\thinspace}
\providecommand{\MR}{\relax\ifhmode\unskip\space\fi MR }
\providecommand{\MRhref}[2]{%
  \href{http://www.ams.org/mathscinet-getitem?mr=#1}{#2}
}
\providecommand{\href}[2]{#2}

\end{document}